\documentclass[11pt]{article}
\usepackage{amsmath, amsfonts, amsthm, amssymb, multicol}
\usepackage{graphicx}
\usepackage{float}
\usepackage{verbatim}

\usepackage[square,sort,comma,numbers]{natbib}
\usepackage{fancyhdr}
 
\numberwithin{equation}{section}

\newtheorem{thm}{Theorem}[section]
\newtheorem{lma}[thm]{Lemma}
\newtheorem{cor}[thm]{Corollary}
\newtheorem{defn}[thm]{Definition}

\newtheorem{prop}[thm]{Proposition}

\renewcommand{\geq}{\geqslant}
\renewcommand{\leq}{\leqslant}

\title{ \vspace{-20mm}Arithmetic patches, weak tangents, and dimension}

\author{Jonathan M. Fraser \& Han Yu}

\begin{document}

\date{}

\maketitle

\begin{abstract}
We investigate the relationships between several classical notions in arithmetic combinatorics and geometry including: the presence (or lack of) arithmetic progressions (or patches in dimensions $\geq 2$); the structure of tangent sets; and the Assouad dimension.

We begin by extending a recent result of Dyatlov and Zahl by showing that a set cannot contain arbitrarily large arithmetic progressions (patches) if it has Assouad dimension strictly smaller than the ambient spatial dimension.  Seeking a partial converse, we go on to prove that having Assouad dimension equal to the ambient spatial dimension is equivalent to having weak tangents with non-empty interior and to `asymptotically' containing arbitrarily large arithmetic patches.

We present some applications of our results concerning sets of integers, which include a weak solution to the Erd\"os-Tur\'an conjecture on arithmetic progressions.
\end{abstract}

\section{Introduction}

\subsection{Arithmetic patches and progressions}

An \emph{arithmetic progression} is a finite subset of $\mathbb{R}$ of the form
\[
P = \left\{ t +   \delta   x : x =0, \dots, k-1   \right\}
\]
for some $t, \delta >0$ and some $k \in \mathbb{N}$.  Here we say $P$ is an arithmetic progression of length $k$ and gap length $\delta$.  Finding arithmetic progressions inside subsets of $\mathbb{R}$ is a topic of great interest in additive combinatorics, number theory and geometry and, in particular, giving conditions which either guarantee the existence of arithmetic progressions or forbid them has attracted much attention.  In the discrete setting, Szemer\'edi's celebrated theorem states that any subset of the natural numbers with positive density necessarily contains arbitrarily long arithmetic progressions. The conclusion of Szemer\'edi's theorem is known to hold for some sets with zero density, however, such as the primes.  This is the content of the Green-Tao Theorem, see \cite{tao}.   In the continuous setting {\L}aba and Pramanik \cite[Theorem 1.2]{laba} showed that a set contains an arithmetic progression of length 3 if it supports a regular measure with certain Fourier decay properties.  Chan, {\L}aba and Pramanik also considered some higher dimensional analogues in \cite{chan}.   In the negative direction, Shmerkin \cite{shmerkin} provided examples of compact Salem sets of any dimension $s \in [0,1]$ which do not contain any arithmetic progressions of length 3, which is in stark contrast to the discrete setting.

The starting point for this work, however, was a recent and elegant observation of Dyatlov and Zahl, which states that any Ahlfors-David regular set of dimension $s<1$ cannot contain arbitrarily long arithmetic progressions, \cite[Proposition 6.13]{zahl}.  Recall that an Ahfors-David regular set of dimension $s$ is a set $F$ such that there exists a constant $C\geq 1$ such that for all $x \in F$ and sufficiently small $r>0$ we have
\[
C^{-1} r^s  \leq \mathcal{H}^s(B(x,r) \cap F)  \leq C r^s
\]
where $\mathcal{H}^s$ is the $s$-dimensional Hausdorff measure.  If a set is Ahfors-David regular of dimension $s$, then most of the familiar notions of dimension used to describe fractal sets coincide and equal $s$. In particular, the Hausdorff, packing, upper box,  lower box, lower, and Assouad dimensions are all necessarily equal to $s$.  For a review of dimension theory and the basic relationships between these dimensions, see \cite{falconer, mattila, Robinson}.  In particular, the Assouad dimension is at least as big as any of the other dimensions listed above.

Our first result, Theorem \ref{main1},  refines the observation of Dyatlov and Zahl by showing that if the Assouad dimension of a set $F \subseteq \mathbb{R}$  is strictly less than 1, then it cannot contain arbitrarily long arithmetic progressions.   This is sharp in the sense that the Assouad dimension cannot generally be replaced by any of the smaller dimensions mentioned above.  We construct examples to show that the converse of this result is not true, but if one relaxes the definition of arithmetic progressions slightly, then one can obtain an `if and only if' statement.  

Our results also hold in arbitrary finite dimensional real Banach spaces, where the notion of arithmetic progression will be replaced by the appropriate analogue: the arithmetic patch.  Let $X$ be a finite dimensional real Banach space with basis $\mathbf{e} = \{e_1, \dots, e_d\}$ for some $d \in \mathbb{N}$.  We denote the associated norm by $\| \cdot \|$ and use the norm to induce a metric, a Borel topology, and Lebesgue measure on $X$, all in the natural way.  For $k \in \mathbb{N}$ and $\delta>0$ we say that a set $P \subset X$ is an \emph{arithmetic patch} of size $k$ and scale $\delta$ (with respect to the basis $\mathbf{e}$) if
\[
P = \left\{ t \, +  \,  \delta  \,  \sum_{i=1}^d  x_i e_i  \ : \ x_1  = 0, \dots, k-1; \dots ; x_d  = 0, \dots, k-1    \right\}
\]
for some $t \in X$.  In particular, $P$ is a discrete set of cardinality $|P| = k^d$ and, in $\mathbb{R}$,  arithmetic patches of size $k$ and scale $\delta$ are precisely the arithmetic progressions of length $k$ and gap length $\delta$.

We say that a set $F \subseteq X$ \emph{contains arbitrarily large arithmetic patches} (with respect to the basis $\mathbf{e}$)  if for all $k \in \mathbb{N}$ there exists a $\delta = \delta(k)$ and an arithmetic patch $P = P(k, \delta)$  of size $k$ and scale $\delta$ (with respect to the basis $\mathbf{e}$) such that $P \subseteq X$.

Since arithmetic patches are finite sets,  containing arbitrarily large arithmetic patches is a (strictly) weaker property than the Steinhaus property, which says that a set contains a scaled copy of every finite configuration of points.  More precisely, we say that a set $F \subseteq X$ satisfies the \emph{Steinhaus property} if for any finite set $P \subseteq X$, there exists $\delta>0$ and $t \in X$ such that $t+\delta P \subseteq F$.  It is a consequence of the Lebesgue density theorem that any Lebesgue measurable set in $\mathbb{R}^n$ with positive Lebesgue measure satisfies the Steinhaus property.

\subsection{Weak tangents}

Weak tangents are tools for capturing the local structure of a  metric space. First we need a suitable notion of convergence for compact sets, which will be given by the Hausdorff metric.  Let $\mathcal{K}(X)$ denote the set of all compact subsets of  $X$, which is a complete metric space when equipped with the Hausdorff metric $d_\mathcal{H}$ defined by
\[
d_\mathcal{H} (A,B) = \inf \{ \delta  \ : \   A \subseteq B_\delta \text{ and } B \subseteq A_\delta \}
\]
where, for any $C \in \mathcal{K}(X)$,
\[
C_\delta \ = \ \{ x \in X \ : \ \| x- y \| < \delta \text{ for some } y \in C \}
\]
denotes the open $\delta$-neighbourhood of $C$.  We will write $B(0,1) \subset X$ to denote the closed unit ball.

\begin{defn}\label{weaktangentdef}
Let $F \in \mathcal{K}(X)$  and $E$ be a compact subset of the unit ball.  Suppose there exists a sequence of similarity maps $T_k: X \to X$ such that $d_\mathcal{H} (E,T_k(F) \cap B(0,1) ) \to 0$ as $k \to \infty$.  Then $E$   is called a \emph{weak tangent} to $F$.
\end{defn}

Recall that a similarity map on a metric space is a bi-Lipschitz map from the space to itself where the upper and lower Lipschitz constants are equal.  In particular, there is a fixed positive and finite  constant depending only on the map such that the map scales all distances uniformly by this constant.

\subsection{Dimension}

The Assouad dimension will be our key notion of metric dimension.  We recall the definition of the Assouad dimension here, but refer the reader to \cite{Robinson, Fraser, Luukkainen} for more details.  For any non-empty totally bounded set $E \subset X$ and $r>0$, let $N_r (E)$ be the smallest number of open sets with diameter less than or equal to $r$ required to cover $E$.  The \emph{Assouad dimension} of a non-empty set $F\subset X$ is defined by
\begin{eqnarray*}
\dim_\text{A} F & = &  \inf \Bigg\{ \  s \geq 0 \  : \ \text{      $ (\exists \, C>0)$ $(\forall \, R>0)$ $(\forall \, r \in (0,R) )$  } \\ 
&\,& \hspace{35mm} \text{ $\sup_{x \in F}N_r\big( B(x,R) \cap F \big) \ \leq \ C \bigg(\frac{R}{r}\bigg)^s$ } \Bigg\}
\end{eqnarray*}
where $B(x,R)$ denotes the closed ball centred at $x$ with radius $R$. It is well-known that the Assouad dimension is always an upper bound for the Hausdorff dimension and upper box dimension, for example.

It is an important problem in quasi-conformal geometry to study how much dimension can be lowered by taking the image of a given set under a  quasi-symmetric mapping.  As such, the \emph{conformal Assouad dimension} of a non-empty set $F\subset X$ is defined by
\[
\mathcal{C} \hspace{-1mm} \dim_\mathrm{A} F \ = \ \inf \{ \dim_\mathrm{A} \phi(F)  :  \phi \text { is quasi-symmetric} \}.
\]
In general, computing, or even effectively estimating, the conformal dimension is  challenging problem.  See \cite{mackaytyson} for more details on this topic and for the precise definition of a quasi-symmetric mapping.  Roughly speaking, a quasi-symmetric mapping is one for which there is control on the relative distortion of two sets.  This is much weaker than than bi-Lipschitz maps for which there is uniform  control on the distortion of the whole space.

One of the most effective ways to bound the Assouad dimension and conformal Assouad dimension of a set from below is to use the weak tangents considered in the previous section. This approach was  pioneered by Mackay and Tyson \cite{mackaytyson}. 

\begin{prop}{ \em \cite[Proposition 6.1.5]{mackaytyson}.}\label{weaktangent}
Let $F, E \in \mathcal{K}(X)$ and suppose $E$ is a weak tangent to $F$.   Then $\dim_\mathrm{A} F \geq \dim_\mathrm{A} E$ and $\mathcal{C} \hspace{-1mm} \dim_\mathrm{A} F \geq \mathcal{C} \hspace{-1mm} \dim_\mathrm{A} E$. 
\end{prop}

\section{Results}

Our first result gives a necessary condition for a set to contain arbitrarily large arithmetic patches.

\begin{thm} \label{main1}
Let $F$ be a non-empty subset of a $d$-dimensional real Banach space $X$.  If $\dim_\mathrm{A} F < d$, then $F$ does not contain arbitrarily large arithmetic patches.
\end{thm}

In fact, Theorem \ref{main1}  will follow from the stronger Theorem \ref{main2} below.  This result is sharp in the sense that the Assouad dimension cannot be replaced by any of the other standard dimension functions.  In particular, consider the countable union of arithmetic patches
\[
E = \bigcup_{n \in \mathbb{N}} \left\{  \delta_n  \sum_{i=1}^d  x_i e_i  \ : \ x_1 =0, \dots, k-1; \dots; x_d =0, \dots, k-1   \right\}
\]
where $\delta_n \to 0$ very quickly (exponentially will do).  Then it is easy to see that the upper box dimension (which is the second largest of the standard dimensions) of $E$  is equal to $0$, despite the set containing arbitrarily large arithmetic patches.

 The converse of Theorem \ref{main1} does not hold, as the following example demonstrates. Let
\[
E_p = \{ 1/n^p : n \in \mathbb{N}\} \subseteq \mathbb{R}
\]
for $p >0$.  It is well-known that $\dim_\mathrm{A}E_p = 1$ for any $p>0$, see for example \cite{garcia, fraseryu}. However, $E_p$ contains no arithmetic progressions of length 3 if $p \geq 3$ is an integer.  Suppose to the contrary that there exists $a, b, c \in \mathbb{N}$ such that
\[
1/a^p-1/b^p = 1/b^p-1/c^p.
\]
It follows that $(bc)^3-2 (ac)^3 +(ab)^3 = 0$, which is not possible.  This was proved by Darmon and Merel \cite{darmon}, and was originally the content of Dene's Conjecture.  Arithmetic progressions of length 3 \emph{are} possible within the squares (and hence within the reciprocals of the squares), but arithmetic progressions of length 4  are not possible.  This last fact was proved by Euler; see \cite{ribet} for a lucid summary of such results.

Essentially, the above examples highlight that the presence of arithmetic patches is more rigid than having maximal dimension.  If we relax the condition slightly, then we can obtain a partial converse.

\begin{defn}\label{AAP}
We say that $F$  \emph{asymptotically contains arbitrarily large arithmetic patches} if, for all $k \in \mathbb{N}$ and $\varepsilon>0$, there exists $\delta>0$ and an arithmetic patch $P$  of size $k$ and scale $\delta$ and a set $E \subseteq F$ such that
\[
d_\mathcal{H} (E,P) \leq \varepsilon \delta.
\]
\end{defn}
In the above definition, $\varepsilon>0$ can be chosen to depend on $k$ if one so wishes.  We also consider an `asymptotic' or `approximate' version of the Steinhaus property. 

\begin{defn}\label{Astein}
We say that $F$ satisfies the  \emph{asymptotic Steinhaus property} if, for all finite sets $P \subset X$ and all $\varepsilon>0$, there exists $\delta>0$,  $t \in X$ and  $E \subseteq F$ such that
\[
d_\mathcal{H} (E, t+\delta P) \leq \varepsilon \delta.
\]
\end{defn}

This time $\varepsilon>0$ can be chosen to depend on $P$ if one so wishes.   Satisfying the asymptotic Steinhaus property clearly implies that a  set asymptotically contains arbitrarily large arithmetic patches, but we shall see that these two properties are actually equivalent.  We can now state our main result.

\begin{thm} \label{main2}
Let $F$ be a non-empty subset of a $d$-dimensional real Banach space $X$.  Then the following are equivalent:
\begin{enumerate}
\item $F$ asymptotically contains arbitrarily large arithmetic patches,
\item $F$ satisfies the  asymptotic Steinhaus property,
\item $F$ has maximal Assouad dimension, i.e. $\dim_\mathrm{A} F = d$,
\item $F$ has maximal conformal Assouad dimension, i.e. $\mathcal{C} \hspace{-1mm} \dim_\mathrm{A} F = d$,
\item $F$ has a weak tangent with non-empty interior,
\item $B(0,1)$ is a weak tangent to $F$.
\end{enumerate}
\end{thm}
 
The hardest part of proving this theorem is establishing the implication $3. \Rightarrow 6.$, which we will do in Section \ref{proof1}.  We will  prove that $6. \Rightarrow 2.$  in Section \ref{proof2}.  Recalling that $2. \Rightarrow 1.$ is trivial, we will prove that $1. \Rightarrow 5.$  in Section \ref{proof3}.  The implication $5. \Rightarrow 4.$ is given by Proposition \ref{weaktangent} and the basic fact that sets with interior points have full conformal Assouad dimension, see \cite{mackaytyson}.  Furthermore, the implication $4. \Rightarrow 3.$ follows immediately  from the definition of conformal dimension, which completes the proof.

Note that Theorem \ref{main1} follows immediately from Theorem \ref{main2} since `asymptotically containing arbitrarily large arithmetic patches' is a weaker property than `containing arbitrarily large arithmetic patches'.  Indeed the following strengthening of Theorem \ref{main1} follows immediately from Theorem \ref{main2}.

\begin{cor}
Let $F$ be a non-empty subset of a $d$-dimensional real Banach space $X$.  If $\dim_\mathrm{A} F < d$, then $F$ does not asymptotically contain arbitrarily large arithmetic patches.
\end{cor}

Up until now, we have been working with a fixed Banach space with a fixed basis.  It is straightforward to show, however, that this is not necessary in Theorem \ref{main2}.

\begin{cor}
Let $F$ be a non-empty subset of a $d$-dimensional real vector space $X$.  If the statements in Theorem \ref{main2} hold for $F$ for a particular choice of basis and norm, then the statements hold for $F$ simultaneously for \emph{any} choice of basis and norm.
\end{cor}

\begin{proof}
Property $3.$~(having full Assouad dimension) only depends on the choice of norm, but, since every norm on a finite dimensional real vector space is equivalent, this property is clearly independent of basis \emph{and} norm.  Moreover, since for a given choice of basis and norm, 1. is equivalent to the other statements, this independence passes to the other statements too.
\end{proof}

In particular,  the property `asymptotically containing arbitrarily large arithmetic patches' is both base and norm independent.  This is not generally true of the more rigid property `containing arbitrarily large arithmetic patches' provided $d \geq 2$ as the following example shows.  (It \emph{is} clearly independent of base and norm if $d=1$.)  Let $X = \mathbb{R}^2$ and consider the standard basis $\{(0,1), (1,0)\}$.  Build a set $E \subset [0,1]^2$ by, for every $k \in \mathbb{N}$, adding an arithmetic patch (in this case a discrete $k \times k$ grid) such that all points land on dyadic rationals.  By definition $E$ contains arbitrarily large arithmetic patches.   Observe that the `direction set' generated by $E$ is countable, i.e.
\[
\text{Dir}(E) \ = \  \left\{ \frac{x-y}{\|x-y\|} \ : \ x,y \in E \right\} \   \subseteq  \ S^1
\] 
is a countable subset of the circle.  Now consider the basis $\{(0,1), e_2\}$, where $e_2 \in S^1 \setminus \text{Dir}(E) \neq \emptyset$.  By the choice of $e_2$, $E$ cannot contain an  arithmetic patch of size 2 with respect to this new basis.

\subsection{Applications to sets of integers}

In this sections we present two applications of our results to sets of integers.  In particular we prove that prime powers asymptotically contain arbitrarily long arithmetic progressions (Corollary \ref{primesszem}) and we also prove a weak version of the Erd\"os-Tur\'an conjecture on arithmetic progressions (Theorem \ref{erdos}).

First of all, the following formulation of Theorem \ref{main2} for sets of integers follows immediately from Theorem \ref{main2} and could be viewed as an asymptotic or approximate version of Szemer\'edi's Theorem.

\begin{cor} \label{coroint}
Let $F \subseteq  \mathbb{Z}$.  Then the following are equivalent:
\begin{enumerate}
\item $\dim_\mathrm{A} F = 1$,
\item $\dim_\mathrm{A} (1/F) = 1$, where $1/F = \{1/n : n \in F\setminus \{0\} \}$,
\item  for all $k \in \mathbb{N}$ and $\varepsilon>0$, there exists $\delta \in \mathbb{N}$ such that one may form an arithmetic progression of length $k$ and gap length $\delta$ by moving elements in $F$ by less than $\varepsilon \delta$.  
\end{enumerate}
\end{cor}

\begin{proof}
Equivalence of 1.~and 2.~follows by the simple fact that Assouad dimension is preserved under the M\"obius transformation $x \mapsto 1/x$, see \cite[Theorem A.10]{Luukkainen}.  The equivalence of 1. and 3. follows immediately from Theorem \ref{main2}.
\end{proof}

In particular, $F \subseteq  \mathbb{Z}$ may have full Assouad dimension, but zero upper Banach density (and so Szemer\'edi's Theorem does not apply directly).  Such examples include: the set $\{n^m : n \in \mathbb{N}\}$ for any $m>1$, the primes, the set of prime powers $\{p^m : p \text{ prime}\}$ for any $m \geq 1$, and `large' sets in the sense of Erd\"os-Tur\'an.  These last examples are particularly interesting because it is known that there does not exist arbitrarily long arithmetic progressions inside powers of integers, let alone powers of primes, and the strict  Erd\"os-Turan conjecture is a wide open problem in number theory.  As such we include the details.

We first require the following technical lemma.
\begin{lma}\label{keyprimeslemma}
For any $m \geq 1$, we have
\[
\dim_\mathrm{A} \,  \{1/ p^m : p \text{ \emph{prime}} \} = 1.
\]
\end{lma}

We will prove Lemma \ref{keyprimeslemma} in Section \ref{primeslemmaproof}.  We actually prove a stronger result, which could be interesting in its own right.  Specifically, we show that there is an infinite subset of the primes  which grows polynomially and whose reciprocals form a sequence with decreasing gaps. This will use recent work on gaps in the primes by Baker, Harman and Pintz \cite{baker}.

\begin{cor}\label{primesszem}
Let $m \geq 1$ and consider the set of $m$th powers of primes $\mathbb{P}^m = \{ p^m : p \text{ prime} \}$.  For all $k \in \mathbb{N}$ and $\varepsilon>0$, there exists $\delta >0$ such that one may form an arithmetic progression of length $k$ and gap length $\delta$ by moving elements in $\mathbb{P}^m$ by less than $\varepsilon \delta$.  
\end{cor}

This result follows immediately from Corollary \ref{coroint} and Lemma \ref{keyprimeslemma}.  One may also obtain higher dimensional analogues of Theorem \ref{primesszem} for sets such as $\mathbb{P}^2 \times \mathbb{P}^2 \subseteq \mathbb{Z}^2 $ or even
\[
\prod_{i=1}^d \mathbb{P}^{m_i} \subseteq \mathbb{R}^d
\]
for any set of reals $m_i \geq 1$, but we leave the precise formulations to the reader.

The Erd\"os-Tur\'an conjecture on arithmetic progressions is a famous open problem in number theory dating back to 1936 \cite{erdos}.  It states that if $F=\{a_n\}_{n \in \mathbb{N}} \subseteq \mathbb{N}$ where $a_n$ is a strictly increasing sequence of positive  integers such that
\begin{equation} \label{largesets}
\sum_{n=1}^\infty 1/a_n \ = \ \infty,
\end{equation}
then $F$ should contain arbitrarily long arithmetic progressions.  Sets of integers which satisfy (\ref{largesets}) are called \emph{large}.  We again begin with a technical lemma concerning Assouad dimension.
\begin{lma}\label{keylargelemma}
If $F \subseteq \mathbb{N}$ is large, then $\dim_\mathrm{A} \, \{ 1/x : x \in F \} = 1$.
\end{lma}

We will prove Lemma \ref{keylargelemma} in Section \ref{largeproof}.  As an immediate consequence of  Corollary \ref{coroint} and Lemma \ref{keylargelemma} we obtain the following weak solution to the Erd\"os-Tur\'an conjecture:

\begin{thm}\label{erdos}
If $F \subseteq \mathbb{N}$ is large, then  for all $k \in \mathbb{N}$ and $\varepsilon>0$, there exists $\delta >0$ such that one may form an arithmetic progression of length $k$ and gap length $\delta$ by moving elements in $F$ by less than $\varepsilon \delta$.  
\end{thm}

\section{Remaining proofs}

\subsection{Full dimension guarantees the unit ball is a weak tangent}  \label{proof1}

In this section we will prove  that $3. \Rightarrow 6.$ in the statement of Theorem \ref{main2}.  Let $X$ be a real Banach space with basis $\{e_1, \dots, e_d\}$.  For a given $R>0$, let $\mathcal{Q}(R)$ be the natural tiling of $X$ consisting of the basic set
\[
\left\{   R  \sum_{i=1}^d  x_i e_i  \ : \ x_1  \in [0,1], \dots, x_d  \in [0,1]  \right\}
\]
and translations thereof by elements of the subgroup $\langle R e_i : i=1, \dots, d \rangle$ of $(X,+)$. By \emph{tiling} we mean that the union over all $Q \in \mathcal{Q}(R)$ is the whole of $X$ and distinct members of $\mathcal{Q}(R)$ intersect in a set of measure zero (a face of dimension $\leq d-1$).   For $r>0$ and a given  $Q \in \mathcal{Q}(R)$, let
\[
M_r(Q) \ = \ \# \left\{ Q' \in \mathcal{Q}(r) : Q' \subseteq Q \text{ and } Q' \cap F \neq \emptyset  \right\}.
\]
If we replace the term $N_r(B(x,R) \cap F)$ in the definition of Assouad dimension with $M_r(Q)$ and then take supremum over $Q \in \mathcal{Q}(R)$, rather than $x \in F$, then it is easily seen that one obtains an equivalent definition.  This is essentially because basic sets $Q \in \mathcal{Q}(R)$ are both contained in a ball of radius comparable to $R$ and contain a ball of radius comparable to $R$.  It is also sufficient to only consider scales $R$ and $r$ which are dyadic rationals.  This will help to simplify our subsequent calculations because the grids formed by the dyadic rationals fit perfectly inside each other.

Let $F \subseteq X$ be such that $\dim_\mathrm{A}F = d = \dim X$.  It follows that for all $n \in \mathbb{N}$ there exists dyadic rationals $r_n, R_n$ satisfying $0<r_n<R_n$ and $\delta(n): =r_n/R_n \to 0$ as $n \to \infty$ and $Q_n \in \mathcal{Q}(R_n)$ such that
\begin{equation} \label{key1}
M_{r_n}(Q_n) \geq \delta(n)^{-(d-1/n)}.
\end{equation}
In order to reach a contradiction, assume that $B(0,1)$ is \emph{not} a weak tangent to $F$. This means that there exists $\varepsilon'>0$ such that for all balls $B$ we have $d_\mathcal{H}(F \cap B , B)>\varepsilon' |B|$, where $|B|$ denotes the diameter of $B$.   This means that we can find a dyadic rational $\varepsilon>0$  such that for all dyadic rationals $R>0$ and all $Q \in \mathcal{Q}(R)$, there exists a $Q' \in \mathcal{Q}(\varepsilon R)$ such that $Q' \subset Q$ and $F \cap  Q' = \emptyset$.  In particular, for all $Q \in \mathcal{Q}(R)$ and all dyadic rationals $r < \varepsilon R$,  we can guarantee that
\begin{eqnarray*}
M_r(Q) &\leq&   \# \{ Q'' \in \mathcal{Q}(r) : Q'' \subseteq Q \}  \ - \   \# \{ Q'' \in \mathcal{Q}(r) : Q'' \subseteq Q' \}   \\ \\
&=& \left(\frac{R}{r}\right)^d -  \left(\frac{\varepsilon R}{r}\right)^d.
\end{eqnarray*}
Consider $Q_n\in \mathcal{Q}(R_n)$ above and note that by the previous statement we know that 
\[
M_{r_n}(Q_n)  \ \leq \  \left(\frac{R_n}{r_n}\right)^d -  \left(\frac{\varepsilon R_n}{r_n}\right)^d  \ = \ \delta(n)^{-d} \left(1 - \varepsilon^d \right)
\]
provided $r_n < \varepsilon R_n$.  This does not contradict (\ref{key1}) alone, which is why we now need to cut out more basic tiling sets of increasingly smaller size.  Consider the $(\varepsilon^{-d} - 1)$ tiling sets from $\mathcal{Q}(\varepsilon R_n)$ which we did not cut out from $Q_n$.  Within each of these, we may cut out one tiling set from $\mathcal{Q}(\varepsilon^2 R_n)$ and, provided $r_n < \varepsilon^2 R_n$, this provides the following improved estimate for $M_{r_n}(Q_n)$:
\begin{eqnarray*}
M_{r_n}(Q_n)  &\leq&  \delta(n)^{-d} \left(1 - \varepsilon^d \right) - (\varepsilon^{-d} - 1) \left(\frac{\varepsilon^2 R_n}{r_n}\right)^d \\ \\
&=&  \delta(n)^{-d} \left(1 - \varepsilon^d -\varepsilon^{2d}(\varepsilon^{-d}-1)\right) .
\end{eqnarray*}

We can continue this process of `cutting and reducing' as long as $r_n <\varepsilon^k R_n$.  Therefore, if we choose $m = m(n) \in \mathbb{N}$ such that $\varepsilon^{m+1} R_n \leq  r_n < \varepsilon^m R_n$, then we finally obtain
\begin{eqnarray*}
M_{r_n}(Q_n)  &\leq& \delta(n)^{-d} \left( 1 -  \sum_{k=1}^m \varepsilon^{kd}(\varepsilon^{-d}-1)^{k-1} \right) \\ \\
&=& \delta(n)^{-d} \left( 1 - \frac{1}{\varepsilon^{-d}-1}\sum_{k=1}^m (1-\varepsilon^{d})^{k} \right) \\ \\
&=& \delta(n)^{-d} \left( 1-\varepsilon^{d} \right)^{m} \\ \\
&\leq& \delta(n)^{-d} \left( 1-\varepsilon^{d} \right)^{\log \delta(n) / \log \varepsilon -1} \\ \\
&=&  \frac{\delta(n)^{-d}\delta(n)^{\log(1-\varepsilon^d)/\log\varepsilon} }{ 1-\varepsilon^{d} }
\end{eqnarray*}
Combining this estimate with (\ref{key1})  yields that for all $n$ we have
\[
\delta(n)^{1/n} \ \leq \   \frac{\delta(n)^{\log(1-\varepsilon^d)/\log\varepsilon} }{ 1-\varepsilon^{d} }.
\]
 This is a contradiction since $\delta(n) \to 0$ as $n \to \infty$.

\subsection{The unit ball being a weak tangent implies the asymptotic \\ Steinhaus property} \label{proof2}

In this section we will prove that  $6. \Rightarrow 2.$  in the statement of Theorem \ref{main2}.  Fix a finite set $P \subseteq X$ with  at least 2 points (otherwise the result is trivial)  and $\varepsilon>0$.  Suppose that $F \subseteq X$ is such that the unit ball $B(0,1)$ is a weak tangent to $F$.  This means that for any $n \in \mathbb{N}$ we can find a ball $B_n$ such that $d_\mathcal{H}(F \cap B_n , B_n) \leq  |B_n|/n$.  Choose
\[
n >  \frac{| P|}{\varepsilon} 
\]
where $|P|$ is the (necessarily positive and finite) diameter of $P$ and let 
\[
\delta = \frac{| B_n |}{|P| } > 0 .
\]
This choice of $\delta$ guarantees that we can find $t \in X$ such that $t+\delta P \subseteq B_n$. One then observes that
\[
\inf_{E \subseteq F}  d_\mathcal{H} (E,t+\delta P) \ \leq  \  d_\mathcal{H}(F \cap B_n , B_n) \   \leq  \  \frac{|B_n|}{n} \   <  \  \frac{\varepsilon | B_n |}{|P|}  \ = \    \varepsilon \delta
\]
which completes the proof.

\subsection{Arbitrarily large patches asymptotically implies weak tangent with interior} \label{proof3}

In this section we will prove  that  $1. \Rightarrow 5.$   in the statement of Theorem \ref{main2}.  More specifically, we will prove that
\[
C \ = \ \left\{    \frac{\sum_{i=1}^d  x_i e_i}{2\sum_{i=1}^d\| e_i\|}  \ : \ x_1  \in [0,1], \dots, x_d  \in [0,1]  \right\} \   \subseteq  \  B(0,1)
\]
is contained in some weak tangent of $F$.  Since $F$  asymptotically contains arbitrarily large arithmetic patches we know that for all $k \in \mathbb{N}$ ($k \geq 2$), there exists $\delta>0$ and an arithmetic patch $P_k$ of size $k$ and scale $\delta$ and a subset $E_k \subseteq F$ such that
\[
d_\mathcal{H} (E_k,P_k) \leq \delta.
\]
Let $T_k$ be the rotation and reflection free similarity which maps the convex hull of $P_k$ to $C$ and consider the sequence
\[
T_k(F) \cap B(0,1) \in \mathcal{K}( B(0,1) ).
\]
Since $ \left(\mathcal{K}( B(0,1) ), d_\mathcal{H}\right)$ is compact we may extract a convergent subsequence with limit $A \subseteq B(0,1)$.  We claim that $C \subseteq A$, which is sufficient to complete the proof.  Indeed, since $T_k(E_k) \subseteq T_k(F) \cap B(0,1)$ and
\begin{eqnarray*}
d_\mathcal{H} (T_k(E_k), C) & = & \frac{d_\mathcal{H} (E_k, T_k^{-1}(C))}{2 \delta (k-1)\sum_{i=1}^d\| e_i\|} \\ \\
& \leq& \frac{d_\mathcal{H} (E_k, P_k)+d_\mathcal{H} (P_k, T_k^{-1}(C))}{2 \delta (k-1)\sum_{i=1}^d\| e_i\|} \\ \\
& \leq& \frac{\delta+\delta\sum_{i=1}^d\| e_i\|}{2 \delta (k-1)\sum_{i=1}^d\| e_i\|} \\ \\
&=&  \frac{1+\sum_{i=1}^d\| e_i\|}{2  (k-1)\sum_{i=1}^d\| e_i\|} \\ \\
&\to& 0
\end{eqnarray*}
as $k \to \infty$, the desired inclusion follows.

\subsection{The dimension of the primes} \label{primeslemmaproof}

In this section we will prove Lemma \ref{keyprimeslemma}.  Computing the Assouad dimension of decreasing sequences $\{ x_n : x_n \searrow 0 \}$ is an interesting problem which has recently been considered in detail by Garc{\i}a, Hare and Mendivil \cite{garcia}.  The problem is greatly simplified if the sequence has \emph{decreasing gaps}, i.e. $x_n-x_{n+1}$ decreases as $n \to \infty$ (or at least eventually decreases).  In fact in this case there is a dichotomy: either the sequence decays subexponentially and the Assouad dimension is 1; or the sequence decays at least exponentially and the Assouad dimension is 0, see \cite[Proposition 4]{garcia}.   We will show that there is a subset of the primes whose reciprocals both `carry the dimension' and have decreasing gaps. 

The main result in \cite{baker} is that for all $n \geq n_0$ there is at least one prime $p_n$ satisfying
\[
n \leq p_n \leq n+O(n^{21/40})
\]
for some effective constant $n_0$. Relying on this result, we may choose an increasing  sequence of primes $p_k$ satisfying:
\[
k^5 \leq p_k \leq k^5+Ck^{21/8}
\]
for some absolute constant $C>0$. We wish to show that the difference between successive gaps
\[
G(k) \ := \ \left( \frac{1}{p_k} - \frac{1}{p_{k+1}} \right) - \left( \frac{1}{p_{k+1}} - \frac{1}{p_{k+2}} \right) 
\]
is positive for sufficiently large  $k$.  We have
\begin{eqnarray*}
 &\,&  \hspace{-15mm} p_{k} p_{k+1} p_{k+2} G(k) \\ \\
&=&   p_{k+1} p_{k+2} +p_{k} p_{k+1}  - 2 p_{k} p_{k+2} \\ \\
&\geq & (k+1)^5 (k+2)^5+  k^5(k+1)^5  -  2  \left(k^5+Ck^{21/8} \right)\left((k+2)^5+C(k+2)^{21/8} \right) \\ \\
&\geq & 30 k^8 - O(k^{61/8}) \  > \ 0
\end{eqnarray*}
for large $k$.  It follows that $P = \{1/p_k \} $ is a decreasing sequence with eventually decreasing gaps, which makes the Assouad dimension straightforward to calculate.  Indeed, since the elements of $P$ have a polynomial lower bound ($1/p_k \geq c/k^5$ for some constant $c$)  it follows from \cite[Proposition 4]{garcia} that  $\dim_\mathrm{A} P = 1$.  Moreover, it follows that $P^m = \{ p^m : p \in P\} \subseteq [0,1]$ is also a decreasing sequence with eventually  decreasing gaps and the terms have a polynomial lower bound (this time $ \geq c'/k^{5m}$).  We may conclude that for any $m \geq 1$ we have
\[
1 \  \geq \  \dim_\mathrm{A} \,  \{1/ p^m : p \text{ prime} \} \  \geq \  \dim_\mathrm{A} P^m \  =  \ 1
\]
which proves the lemma.

\subsection{The dimension of large sets} \label{largeproof}

Let $F =\{a_n\}_{n \in \mathbb{N}} \subseteq \mathbb{N}$ be large with $a_n$ strictly increasing and let $1/F = \{ 1/x : x \in F\}$.  In order to reach a contradiction, assume that $\dim_\mathrm{A} (1/F) <1$, recalling that $F$ and $1/F$ necessarily have equal Assouad dimensions.  It follows that there exists $s \in (0,1)$ and $C>0$ such that for all $k \in \mathbb{N}$ and $r \in (0, 2^{-k})$ we have
\[
N_{r}\left( B(0,2^{-k})  \cap 1/F \right) \ \leq \ C \left( \frac{2^{-k}}{r} \right)^s.
\]
If $a_n, a_{n+1} \leq 2^{k+1}$ for some $k \in \mathbb{N}$ we have
\[
1/{a_n} -  1/a_{n+1}  \geq 4^{-(k+1)}
\]
and so no open set of diameter $4^{-(k+1)}$ can cover the reciprocals of any two distinct points in $ F_k := F\cap [2^{k}, 2^{k+1}]$.  It follows that
\[
|F_k|   \ \leq \ N_{4^{-(k+1)}}\left( B(0,2^{-k}) \cap 1/F \right) \ \leq \ C \left( \frac{2^{-k}}{4^{-(k+1)}} \right)^s \ = \ C \, 2^{s(k+2)}.
\]
Since $F$ is large, we have
\[
\infty \  = \ \sum_{n=1}^\infty 1/a_n \ \leq \ \sum_{k=1}^\infty \  \sum_{n \, : \,  a_n \in F_k} 1/a_n \ \leq \  \sum_{k=1}^\infty |F_k| 2^{-k} \ \leq \  C2^{2s} \sum_{k=1}^\infty 2^{(s-1)k} \ < \ \infty
\]
since $s<1$, which is the desired contradiction.  Observe that the argument in this section directly shows that the Assouad dimension of the primes is 1, however, it does not say anything directly about sets of prime powers.  To deal with these sets, in the previous section we proved that the primes contained a decreasing sequence with decreasing gaps and a polynomial lower bound, which is stronger than having full Assouad dimension.

\vspace{6mm}

\begin{centering}

\textbf{Acknowledgements}

The first named author is  supported by a \emph{Leverhulme Trust Research Fellowship} (RF-2016-500) and the second named author is supported by a PhD scholarship provided by the School of Mathematics in the University of St Andrews.  The authors thank Kenneth Falconer for suggesting we include analogues of the Steinhaus property.
\end{centering}

\begin{multicols}{2}{

\noindent \emph{Jonathan M. Fraser\\
School of Mathematics and Statistics\\
The University of St Andrews\\
St Andrews, KY16 9SS, Scotland} \\

\noindent  Email: jmf32@st-andrews.ac.uk\\ \\

\noindent \emph{Han Yu\\
School of Mathematics and Statistics\\
The University of St Andrews\\
St Andrews, KY16 9SS, Scotland} \\

\noindent  Email: hy25@st-andrews.ac.uk\\ \\
}

\end{multicols}

\end{document}